\documentclass[a4paper,12pt]{amsart}
\usepackage[utf8]{inputenc}
\usepackage{mathtools}
\usepackage{amsmath}
\usepackage{amssymb}
\usepackage{graphicx}
\usepackage{changepage}
\graphicspath{ {images/} }
\usepackage{mathrsfs}
\usepackage{listings}
\usepackage{amsthm}
\usepackage[toc,page]{appendix}
\usepackage[margin=1.1in, headsep=10pt]{geometry}
\usepackage{framed,enumitem}
\usepackage{amsaddr}
\usepackage{dsfont}

\usepackage[colorlinks = true,
linkcolor = blue,
urlcolor  = blue,
citecolor = blue,
anchorcolor = blue]{hyperref}

\newtheorem{thm}{Theorem}[section]
\newtheorem{defn}[thm]{Definition}
\newtheorem{cor}[thm]{Corollary}

\newtheorem{prop}[thm]{Proposition}
\newtheorem{lem}[thm]{Lemma}

\title{Limit shape of subpartition maximizing partitions}

\author{Ivan Corwin and Shalin Parekh}

\begin{document}
	
\maketitle

\begin{abstract}This is an expository note answering a question posed to us by Richard Stanley, in which we prove a limit shape theorem for partitions of $n$ which maximize the number of subpartitions. The limit shape and the growth rate of the number of subpartitions are explicit. The key ideas are to use large deviations estimates for random walks, together with convex analysis and the Hardy-Ramanujan asymptotics. Our limit shape coincides with Vershik's limit shape for uniform random partitions.
\\
\\
\noindent \textit{This note is dedicated to Joel Lebowitz in appreciation for his tremendous and ongoing contributions to the world of statistical physics.}
\end{abstract}

\section{Maximizing the number of subpartitions} \label{sec:1}

Given a partition $\lambda = (\lambda_1 \ge ...\ge \lambda_k)$ of $n$, we can identify it with a $1$-Lipschitz function which is a finite perturbation of $|x|$ by following the Russian convention for drawing it. Specifically, start with the English convention for the Young diagram for $\lambda$ ($\lambda_1$ boxes on the top row, then $\lambda_2$ below it and so on, all justified to line up on the left) and rotate it by $135^\circ$. Then we place this rotated picture immediately adjacent to the graph of the function $x \mapsto |x|$ so that each box has unit length. This defines a $1$-Lipschitz function $g_{\lambda}(x)$ with the property that $g_{\lambda}(x) \ge |x|$ and $g_{\lambda}(x) = |x|$ for large $x$. We also define a rescaled version of $g_{\lambda}$ as $f_{\lambda}(x):= n^{-1/2}g_{\lambda}(n^{1/2}x)$ so that each box has side length $n^{-1/2}$ and area $n^{-1}$ when depicted beneath the graph of $f_{\lambda}$. In particular $\int_{\Bbb R} (f_{\lambda}(x)-|x|)dx = 1$. 
\\
\\
A \textit{subpartition} of a partition $\lambda = (\lambda_1 \ge ... \ge \lambda_k)$ is a partition $\mu = (\mu_1\ge ...\ge \mu_{\ell})$ such that $\ell \le k$ and $\mu_i \leq \lambda_i$ for all $i\le \ell$. Our main result is as follows.

\begin{thm}[Theorems \ref{10} and \ref{9}]\label{mr}
    For each $n$, let $\lambda_n$ denote a partition of $n$ which maximizes the number of subpartitions among all other partitions of $n$. Then the number of subpartitions of $\lambda_n$ grows as $e^{\pi\sqrt{2n/3}-o(\sqrt{n})}$ as $n \to \infty$. Moreover $f_{\lambda_n}$ converges uniformly as $n\to \infty$ to the function $f(x) = \frac{2\sqrt{3}}{\pi}\log\big(2\cosh(\frac{\pi}{2\sqrt{3}}x)\big). $
\end{thm}

\vspace{0.1 in}

The limit shape here is known as \textit{Vershik's curve} and was first described as the limit of uniformly sampled partitions of $n$ in \cite{Ver96}. Our result can be shown by using \textit{large-deviations estimates} for uniformly sampled partitions of $n$ which were found in the follow-up paper \cite{DVZ98}. In particular, to prove Theorem \ref{mr}, first note by the Hardy-Ramanujan asymptotics that the number of subpartitions of any partition of $n$ is bounded above (up to some constant factor) by $e^{\pi \sqrt{2n/3}}$. We let $\mu_n$ be a partition of $n$ which is closest to Vershik's curve (after normalization by $\sqrt n$), among all other partitions of $n$. Fixing $\epsilon>0$, it follows from Theorem 1 of \cite{DVZ98} that for large enough $n$, ``most" partitions of $\lceil (1-\epsilon) n\rceil$ are going to be subpartitions of $\mu_n$, which means that the number of subpartitions of $\mu_n$ is bounded below by $\frac1n e^{\pi \sqrt{2(1-\epsilon)n/3} - o(\sqrt n)}.$ Since $\epsilon$ can be made arbitrarily small, this gives tight bounds on the exponential scale which can then be used (via elementary topological arguments) to show that the maximizing partitions $\lambda_n$ are very close to $\mu_n$ on the $\sqrt n$ scale, so that the $\lambda_n$ also converge to Vershik's curve. 
\\
\\
The main purpose of this note is to exposit the power of large deviations theory in this particular context of partition/subpartition problems. Specifically we are going to give a proof of Theorem \ref{mr}, which is essentially a more rigorous version of the sketch given in the preceding paragraph. However, our exposition is more self-contained and based entirely on foundational principles (specifically we do not use \cite{DVZ98}, but instead rely on the result of Mogulskii \cite{Mog92} which gives a large deviations rate function for the full sample path of a random walk with iid increments, and is arguably a central result of large deviations theory).
\\
\\
We also have the following similar result for $k$-chains of subpartitions, i.e., simply ordered sets of $k$ subpartitions. The ordering may be strict or unstrict; our results do not depend on this convention.

\begin{thm}[Section \ref{sec:6}]\label{1.2}
Let $k\ge 1$, and let $\lambda_n$ denote a partition of $n$ which maximizes the number of $k$-chains of subpartitions, among all other partitions of $n$. Then the number of $k$-chains of subpartitions of $\lambda_n$ grows as $e^{k\pi\sqrt{2n/3} - o(\sqrt{n})}$ as $n \to \infty$. Furthermore $f_{\lambda_n}$ converges uniformly to the same limit shape as in Theorem \ref{mr}.
\end{thm}

We close out this introduction by noting a few questions that may warrant further study. In some cases, there are related results though we do not attempt to make a survey of them. 
\\
\\
One natural question is to consider fluctuations around limit curves, as done in \cite{Yak99, VFY99, VY01, IO03} for instance. For the problem we have considered, this is a bit difficult to phrase since for each $n$ we expect only a few maximizing partitions. On the other hand, if we let $s(\lambda)$ denote the number of subpartitions of $\lambda$, then we may, for $\beta\geq 0$ define a measure on partitions of $n$ with probability of $\lambda$ proportional to $s(\lambda)^\beta$. When $\beta\to \infty$, this measure concentrates on those $\lambda$ which maximize $s(\lambda)$, hence our problem. When $\beta=0$, this measure reduces to the uniform measure on partitions considered by Vershik. While we expect (in particular, based on our arguments in this paper) that the limit shape does not depend on $\beta$, it would be interesting to probe the dependence of $\beta$ on the fluctuations around that shape. It might also be interesting to obtain concentration and large deviations bounds for such measures, as established in \cite{VK85, DVZ98} for instance.
\\
\\
While there are many other types of measures on partitions, one of particular importance is the Plancherel measure. This involves defining the dimension of $\lambda$ to be the number of standard Young Tableaux of that shape. In terms of subpartitions, this is the number of $n$-chains of subpartitions where we restrict that a subpartition cannot equal the partition. The Plancherel measure is then proportional to that dimension squared. For that measure, seminal and independent works of Logan-Shepp \cite{LS77} and Vershik-Kerov \cite{VK77} established a limit shape as $n\to \infty$ now known as the Logan-Shepp-Vershik-Kerov (LSVK) curve. This limit curve is not the same as Vershik's curve. Hence, a natural question is to find a way to interpolate the model so as to find limit shapes which likewise interpolate between these two curves. 
\\
\\
Theorem \ref{1.2} shows that taking $k$-chains for $k$ fixed does not achieve this aim of crossing over between the Vershik and LSVK curves. However, we speculate that taking $k=k(n)=cn^{1/2}$ may result in such a crossover. In fact, this problem can be reduced to a rhombus tiling limit shape problem for which there are some methods which may be useful. Another natural question involves increasing the dimension and considering higher dimensional partitions. In three dimensions, these would correspond with plane partitions, which are also nicely interpreted as rhombus tilings.
\\
\\
\textbf{Acknowledgements:} The authors are thankful to Greg Martin and Richard Stanley who initiated a conversation on MathOverflow two years ago on this question, and in particular to Richard Stanley who posed this question to the first author of this work. I. Corwin was partially supported by a Packard Foundation Science and Engineering Fellowship as well as NSF grant DMS:1811143 and DMS:1664650. S. Parekh was partially supported by the Fernholz Foundation's ``Summer Minerva Fellows" program, as well as summer support from I. Corwin's NSF grant DMS:1811143.
\\
\\
\textbf{Outline:} In Section \ref{sec:2} we will derive exponentially sharp upper bounds for the number of nearest-neighbor paths which stay below a given barrier. In Section \ref{sec:3} we introduce a certain functional which will describe the limit shape and the growth rate of the maximizing partitions; this functional appears naturally from the upper bounds of Section \ref{sec:2}. In Section \ref{sec:4} we prove the limit shape theorem abstractly (without identifying the limit shape explicitly), by using nice convexity properties of the functional defined in Section \ref{sec:3}. In Section \ref{sec:5} we use Lagrange multipliers and Hardy-Ramanujan asymptotics to derive the limit shape explicitly (thus completing the proof of Theorem \ref{mr}). In Section \ref{sec:6} we prove Theorem \ref{1.2}. 

\section{Preliminary upper bounds}\label{sec:2}

First we introduce some notation. Always $I$ will denote a subinterval of $\Bbb Z$ or of $\Bbb R$. The specific type of interval will always be made clear from the context. For a (continuous) function $f:I \to \Bbb R$, we define the \textbf{lower convex envelope} of $f$ to be the supremum of all convex functions which are less than or equal to $f$. Note that this is a convex function, which is also the supremum of a countable number of linear functions which are equal (and in fact tangent, if $I=[0,1]$) to $f$ at certain special points. We also define the \textbf{\textit{decreasing} lower convex envelope} to be the sup of all \textit{decreasing} convex functions less than or equal to $ f$, which is a (weakly) decreasing convex function.
\\
\\
Our first lemma is elementary (albeit tedious to state precisely) and says that the lower convex envelope necessarily optimizes a certain type of convex functional over the set of functions less than a given one.

\begin{lem}\label{0}
    Let $\psi: \Bbb R \to \Bbb R\cup\{+\infty\}$ be a convex function. Let $I$ be the discrete interval $\{a,a+1,...,b\} \subset \Bbb Z$. We let $C(I)$ denote the space of all functions from $I \to \Bbb R$. Define a functional $J: C(I) \to \Bbb R$ by the formula
    $$J(f):=
    \sum_{i\in I\backslash\{a\}} \psi\big(f(i)-f(i-1)\big),
    $$
    Fix some $f\in C(I)$, and let $K_{f}:= \{g\in C(I): g\leq f, g(a)=f(a),g(b)=f(b)\}$. Then one has that $\inf_{g \in K_f} J(g) = J(h),$ where $h$ is the lower convex envelope of $f$. Similarly, if $\bar K_f:= \{g \in C(I): g\le f, g(a)=f(a)\}, $ and if we also assume that $\psi$ achieves its minimum at $0$, then $\inf_{g\in \bar K_f} = J(\bar h)$, where $\bar h$ is the decreasing lower convex envelope of $f$.
\end{lem}

\begin{proof} We will work with $K_f$ rather than $\bar K_f$, briefly indicating the necessary modifications at the end of the proof. The argument is essentially a geometric one which proceeds in two steps.
\\
\\
\textit{Step 1.} Firstly, we show that $J(f) \geq J(h)$ whenever $f(a) = h(a)$, $f(b)=h(b)$, and $h$ is the lower convex envelope of $f$. Let $C:= \{x \in I: f(x)=h(x)\}$. The complement of $C$ is the union of some finite collection of disjoint intervals $\bigcup_n^N (a_n,b_n) \cap \Bbb Z$. On each interval $(a_i,b_i)\cap \Bbb Z$ it is clear from the definition of the lower convex envelope that $h$ is just a linear function, i.e., $h(x) = \frac{x-a_n}{b_n-a_n}f(b_n) + \frac{b_n-x}{b_n-a_n}f(a_n)$ for $x \in [a_n,b_n]$. By Jensen's inequality, one sees that $$\sum_{a_n+1}^{b_n}\psi(f(i)-f(i-1)) \geq (b_n-a_n) \psi\big(\frac{f(b_n)-f(a_n)}{b_n-a_n}\big) = \sum_{a_n+1}^{b_n} \psi(h(i)-h(i-1)).$$ 
This is already enough to prove Step 1, since $f$ coincides with $h$ outside of the $[a_n,b_n]$.
\\
\\
\textit{Step 2.} Secondly, we show that $J(h) \ge J(k)$ whenever $h,k$ are both convex functions with the property that $h(a) = k(a)$, $h(b)=k(b)$, and $h \le k$. To do this, we inductively define a sequence $\{h_j\}_{j=a}^b$ of functions: $h_a=h$, and $$h_{j+1}(x) = \max\{h_j(x), (x-j+1)k(j)+(j-x)k(j-1)\}.$$ In more geometric terms, we are simply taking $h_{j+1}$ to be the maximum of $h_j$ with the ``tangent line" to $k$ at $\{j-1,j\}.$ In particular each $h_j$ is convex, and it follows from convexity of $k$ that $h_b = k$. Thus the claim will be proved if we can show that $J(h_j) \geq J(h_{j+1})$ for all $j\in\{a,...,b-1\}$. But this is clear, because $h_j(x)$ agrees with $h_{j+1}(x)$ except for $x$ in some interval $[u,v]$ where it equals $\frac{x-u}{v-u}h_j(v) + \frac{v-x}{v-u}h_j(u).$ Hence the same argument from Step 1 (using Jensen's inequalty) applies to show $J(h_j) \ge J(h_{j+1})$. This completes the proof of step 2. 
\\
\\
Step 1 and Step 2 easily imply the claim because if $g \le f$ with $g(a) = f(a)$ and $g(b)=f(b)$, and if $h\le k$ are their respective lower convex envelopes then we have that $J(g) \geq J(h) \geq J(k)$, where the first inequality is from Step 1 and the second is from Step 2. 
\\
\\
Now suppose we replace $K_f$ by $\bar K_f$. Let $c \in \{a,...,b\}$ be the point at which $f$ achieves its minimum value. Let $h$ and $\bar h$ denote the lower convex envelope and decreasing lower convex envelope (respectively) of $f$. Note that $h = \bar h$ on $\{a,...,c\}$, and $h(c) = f(c)$, and therefore if $g \le f$ then the above argument gives $\sum_{a+1}^c \psi(g(i)-g(i-1)) \ge \sum_{a+1}^c \psi(\bar h(i)-\bar h(i-1)).$ On the other hand, note that $\bar h(x) = f(c)$ for $x \in \{c,...,b\}$, and thus by assuming that $\psi$ achieves its minimum at $0$, we get that $\sum_{c+1}^b \psi(\bar h(i)-\bar h(i-1)) = \sum_{c+1}^b \psi(0) \leq \sum_{c+1}^b \psi(g(i)-g(i-1)),$ as desired.
\end{proof}

\begin{lem}\label{1}
    Let $f:\{0,...,n\} \to \Bbb R$ with $f(0)=0$. Let $S$ denote a simple symmetric nearest-neighbor random walk on $\Bbb Z$. Also, let $g$ denote the decreasing lower convex envelope of $f$. We also let $\Lambda^*$ be the large deviation rate function associated with $S$, which means that $\Lambda^*$ is the Legendre transform of $\lambda \mapsto \log \Bbb E[e^{\lambda S_1}]$. Then 
    \begin{align}\label{eq:star}
    \Bbb P\big(S_i \leq f(i), \forall i \leq n\big) \leq e^{-\sum_{i=1}^n \Lambda^*(g(i)-g(i-1))}.\end{align}
\end{lem}
\begin{proof}
    The proof uses a standard method for obtaining LDP upper bounds \cite{DZ}. Note that for real numbers $\lambda_1,...,\lambda_n$, and any Borel set $C \subset \Bbb R^n$, \begin{align*}\inf_{ x \in C} e^{\sum_1^n \lambda_i (x_i-x_{i-1})} \Bbb P(S \in C) &\leq \Bbb E[ e^{\sum_1^n \lambda_i(S_i-S_{i-1})} ] = e^{\sum_{i=1}^n \Lambda(\lambda_i)}, \end{align*} where $\Lambda(\lambda) = \log \Bbb E[e^{\lambda S_1}]$ and we impose that $x_0:=0$ in the relevant sum. Rearranging this gives us $$\Bbb P(S \in C) \leq e^{-\inf_{x \in C} \sum_1^n \lambda_i(x_i-x_{i-1}) - \Lambda(\lambda_i)}.$$ Now we optimize over all $\lambda_1,...,\lambda_n$. If we assume that $C$ is compact and convex  we can use the minimax theorem for concave-convex functions \cite{Si58} to interchange the sup over $\lambda$ with the inf over $x$, specifically 
    \begin{align}\label{eq:SC}
    \Bbb P(S \in C) &\leq e^{-\sup_{\lambda \in \Bbb R^n} \inf_{x \in C} \sum_1^n \lambda_i(x_i-x_{i-1}) - \Lambda(\lambda_i)}\\
    \nonumber&\le e^{-\inf_{x \in C} \sup_{\lambda \in \Bbb R^n} \sum_1^n \lambda_i(x_i-x_{i-1}) - \Lambda(\lambda_i)}\\
    \nonumber& \leq e^{-\inf_{x\in C} \sum_1^n \sup_{\lambda \in \Bbb R} \big(\lambda(x_i-x_{i-1}) - \Lambda(\lambda)\big)} \\ 
    \nonumber&= e^{-\inf_{x \in C} \sum_1^n \Lambda^*(x_i-x_{i-1})}.
    \end{align} 
    Now we let $C = \{x \in \Bbb R^n: -i\leq x_i \leq f(i), \forall i\}$, which is clearly compact and convex. Note that $S\in C$ is equivalent to the left-hand side of \eqref{eq:star} (owing to the fact that $S$ only takes $\pm 1$ sized jumps). Applying \eqref{eq:SC} and using Lemma \ref{0} to show that $\inf_{x\in C} \sum_1^n \Lambda^*(x_i-x_{i-1}) = \sum_1^n \Lambda^*(g(i)-g(i-1))$, we arrive at \eqref{eq:star}.
\end{proof}

\begin{cor}\label{2}
    Let $f:\{0,...,n\} \to \Bbb R$ with $f(0)=f(n)=0$, and let $g$ denote the lower convex envelope of $f$ (not the decreasing one). Then the number of nearest neighbor bridges which stay below $f$ (i.e., functions $\gamma: \{0,...,n\} \to \Bbb Z$ such that $\gamma(0)=\gamma(n)=0$, and $|\gamma(i)-\gamma(i-1)|=1$, and $\gamma(i) \leq f(i)$ for all $i$) is bounded above by $ 2^ne^{-\sum_{i=1}^n \Lambda^*(g(i)-g(i-1))}.$
\end{cor}

\begin{proof}
    Let us pick a point $k \in \{0,...,n\}$ at which $g$ attains its minimum value. Note that $g(k)=f(k)$. Note by Lemma \ref{1} the number of nearest neighbor paths of length $k$ starting from 0 and lying below $f|_{\{0,...,k\}}$ is less than or equal to $2^ke^{-\sum_1^k \Lambda^*(g(i)-g(i-1))}$. Similarly the number of nearest neighbor paths of length $n-k$ starting from $0$ and lying below $f|_{\{k+1,...,n\}}$ is less than or equal to $2^{n-k} e^{-\sum_{k+1}^n \Lambda^*(g(i)-g(i-1))}$. Note that the number of bridge paths of length $n$ lying below $f$ is less than the number of pairs of paths $(\gamma,\gamma')$ where $\gamma$ is of the former type and $\gamma'$ is of the latter type. Thus the total number of such bridges is bounded above by product of the two individual upper bounds, which equals $2^n e^{-\sum_1^n \Lambda^*(g(i)-g(i-1))}$.
\end{proof}

An important thing to keep in mind is that the bounds of Propositions \ref{1} and \ref{2} are actually sharp up to some subexponential decay factor (see Section \ref{sec:4}). At an intuitive level, what this says is that if we condition a random walk to stay underneath a fixed barrier, then the path which minimizes the energy of the random walk is none other than the lower convex envelope of that barrier. Another thing to keep in mind is that the bounds of this section hold \textit{uniformly} over all partitions, which makes them a little bit stronger than ordinary LDP upper bounds.

\section{The functional describing the limit shape}\label{sec:3}

For a partition $\lambda$, one recalls the definitions of $f_{\lambda}$ and $g_{\lambda}$ given at the beginning of Section \ref{sec:1}. A $1$-Lipschitz function will always refer to a real-valued function $f$ with the property that $|f(x)-f(y)|\leq |x-y|$, or equivalently $f$ is absolutely continuous and $|f'|\leq 1$.
\\
\\
Let us now estimate (or at least upper bound) the number of subpartitions of a given partition. Each subpartition of a given $\lambda$ can be interpreted as a trajectory of a simple symmetric random walk \textit{bridge} which stays below the graph of $g_{\lambda}$ (or alternatively of $f_{\lambda}$ after rescaling). By Corollary \ref{2}, the number of such bridges can be upper bounded quite easily. Specifically let $h_{\lambda}$ denote the lower convex envelope of $f_{\lambda}$, and let $k$ denote a large enough integer so that $g_{\lambda}(x) = |x|$ whenever $|x| \ge k$. Then by Corollary \ref{2} we know that the number of subpartitions of $\lambda$ (i.e., the number of unit-length random walk bridges which lie in between the graphs of $g_{\lambda}(x)$ and $ |x|$) is upper bounded by \begin{align}
&\;\;\;\;\;2^{2k} e^{-\sum_{i=-k}^{k} \Lambda^*\big(n^{1/2}\big[h_{\lambda}(n^{-1/2}i)-h_{\lambda}(n^{-1/2}(i-1))\big]\big)}\notag \\&= e^{\sum_{-k}^{k} \big[\log 2 - \Lambda^*\big(n^{1/2}\big[h_{\lambda}(n^{-1/2}i)-h_{\lambda}(n^{-1/2}(i-1))\big]\big)\big]} = e^{\sqrt{2n}\int_{\Bbb R} \phi(h_{\lambda}'(x)) dx},\label{U}
\end{align}
where in the final equality we are using the piece-wise linearity of $h_{\lambda}$ and defining $\phi(x) := \log 2 - \Lambda^*(x)$. This function $\phi$ will be very important in the ensuing analysis. In particular, note that $\phi(x)$ is a concave and even function defined on $[-1,1]$ which achieves its maximum value of $\log 2$ at $x=0$, and its minimum of $0$ at $x=\pm 1$.
\\
\\
The functional $f \mapsto \int_{\Bbb R} \phi\circ f'$ appearing in \eqref{U} will describe the optimal rate of growth of the number of subpartitions, as we will show in the following section. Therefore the remainder of this section will be devoted to analyzing this functional. To start, we make the following important definition:

\begin{defn}\label{6}
We define $\mathcal X$ to be the space of all $1$-Lipschitz functions $f: \Bbb R\to \Bbb R$ such that $f(x) \ge |x|$ and furthermore $\int_{\Bbb R} (f(x)-|x|)dx \le 1$. We equip $\mathcal X$ with the topology of uniform convergence on compact sets. Furthermore, we define the functional $F: \mathcal X \to \Bbb R_+$ by $F(h):= \int_{\Bbb R}\phi \circ h'$, where $\phi := \log 2 - \Lambda^*$ and $h'$ is the derivative of $h$.
\end{defn}

A few remarks are in order about this definition.
Firstly, note that $\mathcal X$ is a compact space. Indeed, this is a consequence of Arzela Ascoli: equicontinuity is obvious, and pointwise boundedness follows from the integral condition on elements of $\mathcal X$ combined with the $1$-Lipschitz property (in fact any $f \in \mathcal X$ is bounded above by $x \mapsto \sqrt{x^2+2}$, since this curve is the locus of all points such that the rectangle which has one vertex at that point and another one at the origin, and is also adjacent to the graph of $|x|$, has area exactly $1$).
\\
\\
Secondly, we remark that even though we equipped $\mathcal X$ with the topology of uniform convergence on compact sets, this convergence is actually equivalent to uniform convergence on all of $\Bbb R$. This once again follows from the fact that for all $f\in\mathcal X$ one has that $|x|\le f(x) \leq \sqrt{x^2+2}$, and also because of the fact that $\sqrt{x^2+2}-|x| \to 0$ as $|x|\to \infty$. In particular it is true that $\mathcal X$ is a complete metric space with respect to the uniform metric $$d(f,g) =  \sup_{x\in\Bbb R} |f(x)-g(x)|.$$ The completeness is a consequence of Fatou's Lemma (to ensure that the value of the integral remains $\leq 1$ after taking a limit). This metric will be used very briefly in the proof of a later theorem (\ref{10}).
\\
\\
Thirdly, it is not immediately clear that the integral defining the functional $F(f)$ actually converges for every $f\in \mathcal X$, but this will be taken care of by the following proposition which also highlights the nicest and most important property of $F$, and will crucially be used later.

\begin{prop}\label{3}
The integral defining the functional $F$ converges for every $f \in \mathcal X$. Furthermore, $F$ is upper semicontinuous on $\mathcal X$.
\end{prop}

\begin{proof}
We will prove that if $f_n$ is a family of $1$-Lipschitz functions such that $f_n \to f$ uniformly, then $\limsup_{n \to \infty} F(f_n) \leq F(f)<\infty$. The key difficulty here is that $F$ is defined on functions on the whole real line, which is not compact. The proof will therefore proceed in two steps: first we replace $\Bbb R$ with a large compact interval and prove the upper semicontinuity in this simpler case; second we prove a certain ``tightness" property \eqref{b} for functions in $\mathcal X$ which will simultaneously also show that the integral defining $F(f)$ necessarily converges for all $f\in\mathcal X$.
\\
\\
The first step is to show that for each fixed (large) $A>0$ one has that \begin{equation}\label{a}\limsup_{n \to \infty} \int_{[-A,A]} \phi \circ f_n' \leq \int_{[-A,A]} \phi \circ f'. \end{equation}The proof of this is quite standard, and purely topological (e.g., does not rely on properties of the space $\mathcal X$). Nevertheless we include a proof of \eqref{a} for completeness.
\\
\\
For simplicity, let us replace the interval $[-A,A]$ by $[0,1]$ (the same argument works in the former case with some extra scaling factors).
Let $\mathcal X[0,1]$ denote the space of $1$-Lipschitz functions on $[0,1]$ equipped with the uniform topology. We will show that 
 the functional $G(f):= \int_0^1 \phi\circ f'$ is upper semicontinuous from $\mathcal X[0,1]\to \Bbb R$. To prove this it suffices to write $G$ as the infimum of some collection of continuous functionals. To do this, we consider partitions $\mathcal P = (0\leq t_1 \leq ... \leq t_n \leq 1)$ of $[0,1]$, and we define $G_{\mathcal P}(f):= \sum_1^n (t_i-t_{i-1}) \phi \big( \frac{f(t_i)-f(t_{i-1})}{t_i-t_{i-1}} \big)$. It is then clear that each $G_{\mathcal P}$ is continuous from $\mathcal X[0,1] \to \Bbb R$. We then claim that $G = \inf_{\mathcal P} G_{\mathcal P}$ (where the infimum is taken over all partitions of $[0,1]$) which would prove upper semicontinuity. To prove this equality, first note by Jensen's inequality and concavity of $\phi$ that for all $a<b$ and all $f$ one has $\int_a^b \phi \circ f' \leq (b-a) \phi\big( \frac{f(b)-f(a)}{b-a}\big)$, which proves that $G \leq \inf_{\mathcal P} G_{\mathcal P}$. To prove the other direction, we define the partition $\mathcal P_n$ to be the one consisting of dyadic intervals $[k2^{-n},(k+1)2^{-n})$ with $0\le k \le 2^n-1$. For a $1$-Lipschitz function $f$ let $f_n$ denote the continuous function with $f_n(0)=0$ and whose derivative $f_n'(x)$ takes the constant value $2^n\big(f((k+1)2^{-n})-f(k2^{-n})\big)$ for $x \in [k2^{-n},(k+1)2^{-n})$. Note that $f_n'$ forms a bounded martingale with respect to the dyadic filtration on the probability space $[0,1]$ (i.e., the filtration associated with the nested family of partitions $\{\mathcal P_n\}_n)$. Consequently $f_n'$ converges to $f'$ a.e, and thus $\phi\circ f_n' \to \phi\circ f'$ a.e. Hence by the bounded convergence theorem we conclude that $G_{\mathcal P_n}(f) = \int_0^1 \phi \circ f_n' \to \int_0^1 \phi \circ f' = G(f)$. This shows that $G \ge \inf_{\mathcal P} G_{\mathcal P}$. This proves upper semicontinuity of $G$ and in turn also proves \eqref{a}. 
\\
\\
Now given that \eqref{a} holds, we want to take $A \to \infty$, but this involves a nontrivial interchange of limits and this is where noncompactness of the real line gets in the way. So now we actually need to use special properties of the space $\mathcal X$.
\\
\\
We will show that for every $\epsilon>0$, there exists some $A=A(\epsilon)>0$ (large) so that for all $f \in \mathcal X$ one has that \begin{equation}\label{b}
\int_{\Bbb R\backslash [-A,A]} \phi \circ f' < \epsilon.
\end{equation} Note that together with \eqref{a}, this is enough to complete the proof that $\limsup_n F(f_n) \leq F(f)$. The key here is, of course, that the bound of \eqref{b} is uniform over \textit{all} functions $f \in \mathcal X$. Note that \eqref{b} also shows that $F(f)<\infty$ for all $f\in\mathcal X$.
\\
\\
To prove \eqref{b}, we first note that if $f$ is $1$-Lipschitz, then $f(x)-x$ is necessarily (weakly) decreasing for $x \ge 0$, thus \begin{align}\label{eq:fndiff}1-f(n)+f(n-1) = \big(f(n-1)-(n-1)\big)-\big(f(n)-n\big) \ge 0 \quad\textrm{for } n\geq 1.\end{align} 
The condition that $\int_{\Bbb R} (f(x)-|x|)dx \leq 1$ shows that $\sum_{n\ge 0} f(n)-n \leq 3$ (e.g., via an integral comparison test, since we know $f(n)-n$ is decreasing and $f(0)\le \sqrt{2}<2$). Then for all $N \ge 1$ we find that 
\begin{align*}
&\;\;\;\;\;\sum_{n=1}^N n(1-f(n)+f(n-1))= \sum_{n=1}^N \sum_{k=1}^n (1-f(n)+f(n-1)) \\&= \sum_{k=1}^N \sum_{n=k}^N (1-f(n)+f(n-1)) = \bigg[\sum_{k=1}^N f(k-1)-(k-1)\bigg] - N(f(N)-N),
\end{align*}
where in the last line we used \eqref{eq:fndiff} so that the inner sum telescopes. Since $N(f(N)-N)\ge 0$ we can upper bound the last expression by $\sum_{k \ge 0} (f(k)-k).$ Hence we can let $N \to \infty$ in the preceding expression and we see that 
\begin{align}\label{eq:fk3}
\sum_{n \ge 1} n (1-f(n)+f(n-1)) \leq \sum_{k\ge 0} f(k)-k \leq 3.\end{align}
\\
Appealing to the definition of $\phi(x)$ we see that in $(-1,1)$, $\phi^\prime(x)=-\tanh^{-1}x$ which has logarithmic singluarities at $\pm 1$. Thus, it follows that $\phi$ asymptotically looks like $x|\log x|$ near $x=\pm 1$, i.e., $\lim_{x \to \pm 1} \frac{\phi(x)}{|x\mp 1|\log|x\mp 1|}$ will be a finite nonzero value. Since $|\log x| \leq Cx^{-1/3}$ near $x=0$, this implies that there exists some $C>0$ such that $\phi(x) \leq C(1-|x|)^{2/3}$ for all $x\in [-1,1]$. In particular, for all $A \ge 0$ one has
\begin{align*}
    \sum_{n \ge A} \phi\big( f(n)-f(n-1)\big) &\leq C\sum_{n\ge A} \big(1-f(n)+f(n-1)\big)^{2/3} \\ &\leq C\bigg(\sum_{n \geq A} n^{-2}\bigg)^{1/3} \bigg( \sum_n n(1-f(n)+f(n-1)) \bigg)^{2/3} \\ &\leq C\cdot A^{-1/3}\cdot 3^{2/3}.
\end{align*}
For the second inequality we use the fact that if $a_n$ are nonnegative real numbers such that $\sum_n na_n<\infty$, then by Holder's inequality $\sum_{n\ge A} a_n^{2/3} \leq \big( \sum_{n\ge A} na_n\big)^{2/3} \big(\sum_{n\ge A} n^{-2}\big)^{1/3}$. The final inequality uses the bound derived in \eqref{eq:fk3}, as well as $\sum_{n \ge A} n^{-2} \leq A^{-1}$.
\\
\\
To close out our proof, observe that Jensen's inequality and the concavity of $\phi$ show that $\int_{[n-1,n]} \phi \circ f' \leq \phi(f(n)-f(n-1))$. This, together with the preceding arguments, then shows that $$\int_A^{\infty} \phi\circ f' \leq \sum_{n=A}^{\infty} \phi\big( f(n)-f(n-1)\big) \lesssim A^{-1/3},$$ independently of $f$, which finally proves \eqref{b}.
\end{proof}

At this point it is important to remark that Proposition \ref{3} is \textbf{not} just some technical and otherwise unimportant intermediate step. Really it is where the ``meat" of the proof of the limit shape (Theorem \ref{mr}) really lies. Specifically, the important thing here is the second half of the proof where we prove a type of ``tightness" estimate \eqref{b} for functions in $\mathcal X$. In terms of partitions, what it really shows (in an equivalent formulation) is that the sequence of partitions maximizing the number of subpartitions, stays bounded on the $n^{1/2}$ scale, i.e., that the sequence $f_{\lambda_n}$ from Theorem \ref{mr} does not lose any mass in the limit (meaning that any subsequential limit $f$ of $f_{\lambda_n}$ satisfies $\int(f(x)-|x|)dx=1$). We remark that the bound $A^{-1/3}$ appearing at the end of the proof may actually be improved optimally to $\frac{\log A}{A}$, but this is slightly more difficult.
\\
\\
As a corollary of Proposition \ref{3}, we can combine it with compactness of the space $\mathcal X$ in order to obtain the following key result.

\begin{cor}\label{4}
The functional $F$ from Definition \ref{6} admits a maximum $M(F)$ on the space $\mathcal X$. There is a unique function $f$ at which the maximum is attained and this maximizer $f$ is a convex and symmetric function (i.e. $f(x) = f(-x)$) and moreover $\int_{\Bbb R} (f(x)-|x|)dx = 1$.
\end{cor}

\begin{proof}
Any upper semicontinuous function on a compact space achieves its maximum.
\\
\\
The uniqueness of the maximizer is a concavity property. Specifically we note that $\phi$ is a strictly concave function, meaning $\phi((1-t)a+tb) > (1-t)\phi(a) + t\phi(b)$ whenever $t\in(0,1)$ and $a\ne b$. This then easily implies that $F((1-t)f+tg) > (1-t)F(f)+tF(g) $ for $t\in (0,1)$ and $f \ne g$. Clearly this rules out the existence of two distinct maxima.
\\
\\
Symmetry is another consequence of concavity. Specifically, if the maximizer $f$ was not symmetric, then we can define its reflection $f_s(x):=f(-x)$. Clearly $F(f_s) = F(f)$ and thus if $f \neq f_s$ then as above we have that $F(\frac12 f+\frac12 f_s) > \frac12 F(f) + \frac12 F(f_s) = F(f),$ which is a contradiction.
\\
\\
Let $f$ be the maximizer. To prove that $\int(f(x)-|x|)dx=1$, suppose that this integral took some value $\alpha<1$. Then we let $h(x) = \alpha^{-1/2}h(\alpha^{1/2}x)$. Clearly $\int (h(x)-|x|)dx = 1$, and $h$ is $1$-Lipschitz. Moreover a simple substitution reveals that $F(h) = \alpha^{-1/2}F(f) >F(f)$ which is a contradiction.
\\
\\
To prove convexity of $f$, suppose (for contradiction) that $a,b$ are two points of $\Bbb R$ such that there is a linear function $\ell$ equal to $f$ at both $a$ and $b$, and such that $\ell<f$ on $(a,b)$. We define $h$ to be equal to $f$ on $\Bbb R\backslash [a,b]$, and equal to $\ell$ on $[a,b]$. Then by Jensen's inequality one has that $\int_a^b \phi \circ f' < (b-a) \phi\big( \frac{f(b)-f(a)}{b-a}\big) = \int_a^b \phi \circ h'$, which means that $F(f)<F(h)$; a contradiction. This completes the proof.
\end{proof}

\section{The limit shape Theorem}\label{sec:4}

Note that in \eqref{U} we already proved that for any sequence $\lambda_n$ of partitions of $n$, the number of subpartitions is bounded above by $e^{\sqrt{2n}M(F)}$ where $M(F)$ is the maximum value of the functional $F$ from above. A natural question is whether there exists a sequence of partitions for which the number of subpartitions actually grows at this optimal rate. It turns out that the answer is yes (up to some subexponential factor which is irrelevant), which retrospectively justifies why we performed such an in-depth analysis of the functional $F$ in the first place.

\begin{prop}\label{5}
There exists a sequence of partitions $\mu_n$ of $n$ such that the number of subpartitions of $\mu_n$ actually grows as $e^{\sqrt{2n}M(F) - o(\sqrt{n})}$ as $n \to \infty$. 
\end{prop}

The key behind proving this proposition is Mogulskii's theorem \cite{Mog92}, which is really the primary underlying idea behind this entire work. This result essentially says that the bound in \eqref{U} (and also in Propositions \ref{1} and \ref{2}) is actually \textit{sharp} (again, up to some subexponential factor which is not relevant to us). But before getting to the proof, let us first prove the following important corollary.

\begin{thm}[Limit shape theorem]\label{10}
Let $\lambda_n$ and $f_{\lambda_n}$ be as in Theorem \ref{mr}. As $n \to \infty$, the sequence $f_{\lambda_n}$ converges uniformly to the unique maximizer $f_{\max}$ of the functional $F$ from Definition \ref{6}. 
\end{thm}

\begin{proof}
Let $s(\lambda_n)$ denote the number of subpartitions of $\lambda_n$, and let $h_{\lambda_n}$ denote the lower convex envelopes of $f_{\lambda_n}$. By Proposition \ref{5} and equation \eqref{U} we have that $$e^{\sqrt{2n}M(F) - o(\sqrt{n})} \leq s(\lambda_n) \leq e^{\sqrt{2n} F(h_{\lambda_n})} \leq  e^{\sqrt{2n}M(F)},\;\;\;\;\;\;\;\;\;\text{as}\;\; n\to \infty,$$ which means that $M(F)-o(1) \leq F(h_{\lambda_n}) \leq M(F)$ as $n \to \infty$.
\\
\\
Thus we see that $F(h_{\lambda_n}) \to M(F)$ as $n \to \infty$. This is already enough to imply that $h_{\lambda_n} \to f_{\max}$ uniformly on compact sets as $n \to \infty.$ Indeed it is true that for every $\epsilon>0$ there exists $\delta>0$ such that (for all $f \in \mathcal X$) $F(f)>M(F)-\delta$ implies that $d(f,g)<\epsilon$ (here $d$ denotes the metric on $\mathcal X$ which was specified following Definition \ref{6}). If this was not the case then we can choose an $\epsilon>0$ such that $\sup_{d(f_{\max},g)\ge \epsilon} F(g) = M(F)$. But the space $\mathcal A$ of $1$-Lipschitz functions $g$ such that $d(f_{\max},g)\ge \epsilon$ is again a compact subset of $\mathcal X$ (being a closed subset of $\mathcal X$). Furthermore $F$ is still an upper semicontinuous function on $\mathcal A$, hence it achieves its maximum value which we already know is $M(F)$. Then there exists some $g_{\max} \in \mathcal A$ such that $F(g_{\max}) = M(F)$, which clearly contradicts uniqueness of the maximizer since $d(f_{\max},g_{\max})\ge \epsilon$ by construction.
\\
\\
So we have proved that the convex envelopes $h_{\lambda_n}$ (though not necessarily the functions $f_{\lambda_n}$ themselves) converge uniformly to $f_{\max}$. Note that since $f_{\lambda} \geq h_{\lambda}$ (by definition of the convex envelope) we have \begin{align*} \int_{\Bbb R} |f_{\lambda_n} - h_{\lambda_n}| &=\int_{\Bbb R} (f_{\lambda_n}(x)-h_{\lambda_n}(x)) dx \\ &= \int_{\Bbb R} \big((f_{\lambda_n}(x)-|x|) - (h_{\lambda_n}(x) - |x|)\big)dx \\ &= 1- \int_{\Bbb R} (h_{\lambda_n}(x)-|x|)dx.
\end{align*}
Now $h_{\lambda_n}$ converges to $f_{\max}$ and by Corollary \ref{4} we know that $\int(f_{\max}(x)-|x|)dx = 1$, therefore by applying the preceding calculation and then Fatou's Lemma, we see that $$\limsup_n \int_{\Bbb R} |f_{\lambda_n}-h_{\lambda_n}| = 1-\liminf_n \int_{\Bbb R} (h_{\lambda_n}(x)-|x|)dx \leq 1-\int (f_{\max}(x)-|x|)dx = 0.$$ Therefore $\|f_{\lambda_n}-h_{\lambda_n}\|_{L^1(\Bbb R)} \to 0$, and since all functions are $1$-Lipschitz, this $L^1$ convergence also implies uniform convergence.
\end{proof}

Although this abstractly proves convergence to some limit shape, we still do not know anything about what the limit shape looks like geometrically. For instance is it bounded, and if so, is it a triangular shape or something more complicated? This question will be addressed in the following section.
\\
\\
Let us now start to get to the proof of Proposition \ref{5}. As mentioned before, the key is the following result, which essentially gives matching lower bounds to the upper bounds which we gave in Section \ref{sec:2}. A proof may be found in Theorem 5.2.1 of \cite{DZ} or in the original paper \cite{Mog92}.

\begin{thm}[Mogulskii 1992] Let $\mu_n$ denote the law on $C[0,1]$ of $(\frac1n S_{nt})_{t \in [0,1]}$ where $S$ is any i.i.d. random walk (whose increment distribution has exponential moments), and the values of $S$ at non-integer points are understood to be linearly interpolated from the two nearest integer points. Then $\mu_n$ satisfies an LDP with rate $n$ and good rate function $$I(f) = \int_0^1 \Lambda^* \circ f',$$ where $\Lambda^*$ denotes the Legendre transform of $\lambda \mapsto \log \Bbb E[e^{\lambda S_1}]$, and the integral is meant to be understood as $+\infty$ if $f$ is not absolutely continuous.
\end{thm}

It should be noted that Mogulskii's result is a vast strengthening of Cramer's theorem (from just the endpoint of an iid sample path to its entire history), in the same way that Donsker's invariance principle for iid random walks is a strengthening of the classical central limit theorem. Finally we are ready to prove Proposition \ref{5}.

\begin{proof}[Proof of Proposition \ref{5}]Let $f_{\max}$ be the maximizer from Corollary \ref{4}. We choose a sequence $\mu_n$ of partitions of $n$ such that $f_{\mu_n}$ converges uniformly to $f_{\max}$. This can be done as follows. First we construct an intermediate partition $\tilde \mu_n$ by putting boxes of side length $n^{-1/2}$ beneath the graph of of $f_{\max}$ until no more boxes can be put in such a way that the graph of $f_{\tilde \mu_n}$ remains below that of $f$. Since $f_{\tilde \mu_n} \leq f$, one notices that $\tilde \mu_n$ will not actually be a partition of $n$ but rather of some number $k(n) \leq n$. However, it is true that $|f_{\tilde \mu_n} - f_{\max}| \leq Cn^{-1/2}$ for some constant independent of $n$ (otherwise more boxes could be added to $\tilde \mu_n$ without eclipsing the graph of $f_{\max}$). Now we can define $\mu_n$ to be equal to $\tilde \mu_n$ but with the remaining $n-k(n)$ boxes added to the first column of $\tilde \mu_n$. This will not change the limiting function $f_{\max}$.
\\
\\
We define $f_{\delta}(x):=\max\{|x|, f_{\max}(x)-\delta\}$, and we define the \textit{support} of $f_{\delta}$ to be the set of $x$ where $f_{\delta}(x)>|x|$ (this is an interval centered at $0$, by convexity and symmetry of $f_{\max}$). Note that for large enough values of $n$, the $\delta/2$ neighborhood of $f_{\delta}$ lies strictly below $f_{\mu_n}$ on the support of $f_{\delta}$ (this is because $f_{\mu_n}\to f_{\max}$ uniformly). We are now going to consider nearest-neighbor (random walk) paths of grid-size $n^{-1/2}$ which lie in between the graphs of $f_{\delta/2}$ and $f_{3\delta/2}$. Such a path will be called $(\delta,n)$-admissible. Let $k=k(n,\delta)$ denote the positive integer such that $n^{-1/2}k = \text{argmin}_{y \in \frac1{\sqrt{n}}\Bbb Z} |y-a|$ where $a=a(\delta):=\inf\{x>0: f_{\delta}(x)=x\}$.
\\
\\
Note by Mogulskii's Theorem that the number of $(\delta,n)$-admissible paths terminating on the vertical axis (i.e., nearest-neighbor functions $\gamma:n^{-1/2}\Bbb Z_{\le 0} \to n^{-1/2}\Bbb Z$ where $\Bbb Z_{\le 0}$ denotes non-positive integers) is greater or equal to $2^k e^{-\sqrt{2n}\int_{-n^{-1/2}k}^0 \Lambda^* \circ f_{\delta}' - o(\sqrt{n})} = e^{\sqrt{2n} \int_{-\infty}^0 \phi\circ f_{\delta}' - o(\sqrt{n})} $, as $n \to \infty$ (with $\delta$ fixed).
\\
\\
Now we notice that two independent such random walks which are started from $(-n^{-1/2}k,n^{-1/2}k)$ and conditioned to stay between $f_{\delta/2}$ and $f_{3\delta/2}$ have probability at least $\frac{1}{\delta \sqrt n}$ of terminating at the same point. Indeed, this is because there are at most $\delta \sqrt n$ possible points $\{x_i\}_{i=1}^{\delta\sqrt n}$ at which such a walk can terminate (because the grid size is $n^{-1/2})$, and if $p_i$ is the probability of terminating at point $x_i$, then by Cauchy-Schwarz one finds that $1 =\sum_1^{\delta \sqrt n} p_i \leq (\sum_i p_i^2)^{1/2} (\delta \sqrt n )^{1/2}$, and because the probability of two independent such walks terminating at the same point equals $\sum_i p_i^2$.
\\
\\
Now, a random walk \textit{bridge} of grid size $n^{-1/2}$ which lies between $f_{\delta/2}$ and $f_{3\delta/2}$ (which defines a subpartition of $\mu_n$ for large enough $n$) can be viewed as the concatenation of a pair of these random walk paths started from $(-n^{-1/2}k,n^{-1/2}k)$ terminating at the same point on the vertical axis (note here that we are using the property that $f_{\delta}(x) = f_{\delta}(-x))$. By the observations of the preceding two paragraphs, the number of such pairs is bounded below by $\frac{2}{\delta \sqrt{n}} \big(e^{\sqrt{2n} \int_{-\infty}^0 \phi\circ f_{\delta}' - o(\sqrt{n})}\big)^2.$ The prefactor $\frac2{\delta\sqrt n}$ may be absorbed into the $o(\sqrt n)$ term in the exponent, giving a lower bound of $e^{\sqrt{2n} F(f_{\delta})-o(\sqrt{n})}.$ The $o(\sqrt n)$ term may depend on $\delta$ but this is not a problem.
\\
\\
Since this lower bound holds true for arbitrary $\delta>0$, the claim now follows if we can show that $F(f_{\delta}) \to F(f_{\max})$ as $\delta \to 0$. To do this, note that $f_{\delta}' \to f_{\max}'$ pointwise (trivially by the definition of $f_{\delta}$). Thus by Fatou's Lemma and maximality of $F(f_{\max})$ it is true that $F(f_{\max}) \leq \liminf_{\delta \to 0} F(f_{\delta}) \le \limsup_{\delta} F(f_n) \leq \max_g F(g)= F(f_{\max})$, which completes the argument.
\end{proof}

\section{Characterizing the limit shape}\label{sec:5}

So far, many of our methods could have been used for more general types of models than the simple symmetric random walk (replacing $\phi$ with a more general concave function). We now move onto trying to find the limit shape $f_{\max}$ exactly, which will involve working with specific details of the function $\phi$, and thus most of the subsequent arguments and analysis will be specialized just to the case of the simple random walk. In particular we will show that $f_{\max}$ has, up to scaling and centering, the shape of the curve $x \mapsto \log \cosh x$. In particular it is not just the triangular function $x\mapsto \max\{1,|x|\}$, nor is it the Vershik-Kerov curve. It is, in fact, the Vershik curve which is the limit shape of \textit{uniformly} sampled partitions of $n$ \cite{Ver96}.
\\
\\
Since $f_{\max}$ is an even convex function there exists a \textit{maximal} interval $(-a_{max},a_{max})$ (which we will henceforth refer to as the \textbf{support} of $f_{\max}$) on which $f(x)>|x|$. This interval is the interior of the largest closed interval containing the support (in  the usual sense) of the second distributional derivative $f_{\max}''$ (which is a nonnegative Borel measure since $f_{\max}$ is convex). Note that it is possible that $a_{max}=+\infty$, and in a moment we will show that this is indeed the case.
\\
\\
Let $\psi$ be a smooth function with support contained in $(-a_{max},a_{max})$, such that $\int_{\Bbb R} \psi = 0$. Then we claim \begin{equation}\label{7}
    \int_{\Bbb R} (\phi' \circ f_{\max}') \cdot \psi' = 0.
\end{equation}
Indeed, one easily checks that $\lim_{\epsilon \to 0} \epsilon^{-1} \big(F(f_{\max}+\epsilon\psi)-F(f_{\max})\big) = \int_{\Bbb R} (\phi'\circ f_{\max}')\cdot \psi'$. However, since $\int \psi = 0$ and since the support of $\psi$ is contained in the support of $f_{\max}$, it follows that for $\epsilon$ in a small enough neighborhood of zero, the function $f_{\max}+\epsilon \psi$ is an element of $\mathcal X$, and thus $F(f_{\max}+\epsilon \psi)\leq F(f_{\max})$. Hence if $\lim_{\epsilon \to 0} \epsilon^{-1} \big(F(f_{\max}+\epsilon\psi)-F(f_{\max})\big)$ exists then it must equal zero, proving \eqref{7}.
\\
\\
Now if $h: [-a,a] \to \Bbb R$ is any measurable function such that $\int h\cdot \psi' = 0$ for every function $\psi \in C_c^{\infty}$ with $\int \psi = 0$, then this precisely means that the distributional derivative of $h$ is orthogonal (with respect to the $L^2$ pairing) to all except the constant functions. In particular it means that $h'$ is itself a constant function. Applying this principle to $h:= \phi' \circ f_{\max}'$, we see that $\phi'(f_{\max}'(x)) = \beta x + C$ for some $\beta, C \in \Bbb R.$ But $\phi$ and $f_{\max}$ are even functions, so $\phi'\circ f_{\max}'$ is and odd function, and thus $C=0$. Now recall that $\phi = \log 2 -\Lambda^*$ where $\Lambda^*$ is the Legenrde transform of $\Lambda(x) = \log \cosh x$. This implies that $\Lambda' \circ \phi'$ is the negative of the identity function on $[-1,1]$. In particular $\phi'(f_{\max}'(x)) = \beta x$ implies that $-f_{\max}'(x) = \Lambda'(\beta x)$, which in turn implies that $f_{\max}(x) = - \frac1{\beta} \log \cosh(\beta x)+D$ for all $x$ in the support of $f_{\max}$. Here $D$ is some constant of integration. Of course, we know that $f_{\max}$ is convex, which implies $\beta \le 0$. Thus, by renaming $\beta$ to be $-\beta$ we have proved the following.

\begin{prop}\label{8}
There exists some $\beta_{max}\ge 0$ and some $D_{max} > 0$ such that for every $x \in (-a_{max},a_{max})$ one has that $f_{\max}(x) = \frac1{\beta_{max}}\log \cosh (\beta_{max} x) +D_{max}$. 
\end{prop}

In the possibility where $\beta_{max} = 0$, the statement of the above Proposition is of course nonsensical, but (because the condition $\int (f_{\max}(x)-|x|)dx=1$ determines $D_{max}$ uniquely from $\beta_{max}$) it is meant to be interpreted in the sense that $f_{\max}(x) = D_{max}$ on its support, meaning that the limit shape would be the triangular function $x\mapsto \max\{1,|x|\}$. We will rule out this possibility shortly.
\\
\\
Our next goal is to find out whether or not $a_{max}<+\infty$, i.e., whether the limit shape is something compact or not. The next theorem tells us that the answer is no.

\begin{thm}\label{9}In the notations of Proposition \ref{8}, $a_{max} = +\infty$, $\beta_{max} = \frac{\pi}{2\sqrt{3}},$ $D_{max} = \frac1{\beta_{max}}\log 2$, and $F(f_{\max}) = \pi/\sqrt{3}$. In particular, $f_{\max}(x) = \frac{2\sqrt{3}}{\pi}\log\big(2\cosh(\frac{\pi}{2\sqrt{3}}x)\big).$
\end{thm}

\begin{proof}The key will be to use the Hardy-Ramanujan asymptotics together with the identity \begin{align}\label{f}\int_0^{\infty} \log(1+e^{-2x})dx = \int_0^{\infty} \sum_{n \ge 1} (-1)^{n+1}\frac{e^{-2nx}}{n}dx = \sum_{n \ge 1} \frac{(-1)^{n+1}}{2n^2}=\frac{\pi^2}{24} .\end{align}
Here we Taylor expanded the logarithm and then used the identity $\sum_{n\ge 1} n^{-2} = \frac{\pi^2}{6}$ and its corollaries: $\sum_{n\;even} n^{-2} = \frac{\pi^2}{24}$ and $\sum_{n\;odd} n^{-2} = \frac{3\pi^2}{24}$.
\\
\\
We now recall the Hardy-Ramanujan asymptotics \cite{HR} for the partition numbers. Specifically, if $p(n)$ denotes the number of partitions of $n$, then $p(n) = e^{\pi\sqrt{2n/3}-o(\sqrt{n})}$ as $n \to \infty$. Notice that $p$ is an increasing function of $n$, and every subpartition of a partition of $n$ is a partition of some integer $i\le n$. Thus the number of subpartitions of any given partition $\lambda$ of $n$ is upper bounded by $\sum_{i=0}^n p(i) \leq (n+1)p(n) = (n+1) e^{\pi\sqrt{2n/3}-o(\sqrt{n})}$. The prefactor $(n+1)$ may be absorbed into the $o(\sqrt n)$ term in the exponent, and thus by Proposition \ref{5} we conclude that $F(f_{\max}) \leq \pi/\sqrt{3}$.
\\
\\
Now, let $f(x):= \alpha^{-1/2} \log(2\cosh (\alpha^{1/2} x))$, where $\alpha := \int_{-\infty}^{\infty} (\log(2\cosh x)-|x|)dx = 2\int_0^{\infty} \log(1+e^{-2x})dx = \frac{\pi^2}{12}$ by \eqref{f}. Note that $f$ is $1$-Lipschitz (because it has derivative given by $\tanh(\alpha^{1/2}x)$ which is bounded in absolute value by $1$), and also note (by substituting $u=\alpha^{1/2}x$) that $\int_{\Bbb R} (f(u)-|u|)dx = 1$ so that $f \in \mathcal X$. Now we claim that $F(f) = \pi/\sqrt{3} = 2\alpha^{1/2}$, which would indeed prove that $f=f_{\max}$. To prove this, note that $F(f)= \alpha^{-1/2} \int_{\Bbb R} \phi(\tanh u)du,$ so that proving that $F(f) = 2\alpha^{1/2}$ now amounts to showing that $\int_{\Bbb R} \phi(\tanh u)du = 2\alpha$. In other words, we want to show 
\begin{equation}\label{g}
    \int_0^{\infty} \phi(\tanh u)du = 2\int_0^{\infty} \log(1+e^{-2u})du.
\end{equation}
One readily checks that $\phi(\tanh u) = \log(e^u+e^{-u}) -u\tanh u,$ from which proving \eqref{g} amounts to checking that $\int_0^{\infty} \big( \log(e^u+e^{-u}) -2u+u\tanh u\big)du = 0$. But the integrand here has an explicit antiderivative given by $u \log(e^u+e^{-u}) -u^2$, which is readily seen to evaluate to zero at both $u=0$ and as $u \to \infty$. This proves \eqref{g}, which finally shows that $f=f_{\max}$.
\end{proof}

A further direction of study is to try to gain more precise asymptotics on the exact number of subpartitions of the maximizing sequence. Specifically we would like to find precise asymptotics on the $o(\sqrt n)$ term in the optimal growth rate $e^{\pi\sqrt{2n/3}-o(\sqrt n)},$ and we believe this can be done using more precise large deviations estimates. A similarly difficult ``local" asymptotic problem would be to find the rate at which the side lengths go to $\infty$ (note that Theorem \ref{9} merely proves that they grow faster than $\sqrt{n}$).

\section{Extension to $k$-chains of subpartitions}\label{sec:6}

We now extend the limit shape theorem to the case of partitions which maximize the number of $k$-chains of subpartitions, which will prove theorem \ref{1.2}. Since the proof is not significantly more complicated, we briefly indicate the changes which need to be made at each stage of the argument.
\\
\\
First we address the necessary modifications in Section \ref{sec:2}. In the notation of Corollary \ref{2}, consider $k$-chains $\gamma_k\leq... \leq \gamma_2 \leq \gamma_1 \leq f$ of nearest-neighbor bridges which stay below $f$. Then (by viewing the chain as just a $k$-tuple of paths and disregarding the ordering) the same corollary says that the number of these $k$-chains is bounded above by 
$$\bigg( 2^n e^{-\sum_1^n \Lambda^*(g(i)-g(i-1))} \bigg)^k.$$
Then, in equation \eqref{U} at the beginning of Section \ref{sec:3}, this bound will tell us that for a given partition $\lambda$ of $n$, the number of $k$-chains of subpartitions of $\lambda$ (i.e., $k$-chains of random walk bridges of grid size $n^{-1/2}$ which are nestled in between the graphs of $f_{\lambda}(x)$ and $|x|$) is upper-bounded by 
\begin{equation}\label{13}e^{k\sqrt{2n} F(h_{\lambda}')},
\end{equation}
where as usual $h_{\lambda}$ is the lower convex envelope of $f_{\lambda}$, and $F$ is the functional of Definition \ref{6}.
\\
\\
Hence, all that is left to do is to show that the upper bound \eqref{13} is actually sharp up to the exponential scale (after replacing $F(h_{\lambda}')$ with $M(F) = \pi/\sqrt{3}$ there). The way to do this is by modifying the proof of Proposition \ref{5} to lower bound the number of ensembles of $k$ distinct paths staying below the graph of $f_{\max}$. In the notation of that proof, we consider ensembles (implicitly depending on $n$) of nearest neighbor bridges $(\gamma^i)_{i=1}^k$ from $n^{-1/2}\Bbb Z \to n^{-1/2}\Bbb Z,$ with the property that $f_{i\delta}\le \gamma_i \le f_{(i+1)\delta}$ for each $1\le i \le k$. Clearly each such ensemble defines a $k$-chain of subpartitions of $\mu_n$. Moreover the number of such $k$-chains is merely the product over $i\in\{1,...,k\}$, of the individual number of paths lying between the graphs of $f_{i\delta}$ and $f_{(i+1)\delta},$ and we already know a good individual lower bound from the proof of Proposition \ref{5}. Specifically, we can lower bound this number of $k$-chains by 
$$\prod_{i=1}^k \big( e^{\sqrt{2n} F(f_{(i+\frac12)\delta})-o(\sqrt n)}\big) = e^{\sqrt{2n}\sum_{i=1}^k F(f_{(i+\frac12)\delta}) - o(\sqrt n)}.$$
As we already showed in the proof of Proposition \ref{5}, $F(f_{\eta}) \to 0$ as $\eta \to 0$, which means (by making $\delta$ close to $0$) that we can actually lower bound the maximal number of $k$-chains of subpartitions by $e^{k\sqrt{2n}M(F) - o(\sqrt{n})}$, as $n \to \infty$. This already proves Theorem \ref{1.2}. We remark here that the proof does \textit{not} rely on whether or not the $k$-chains are strictly ordered or not, so the statement of Theorem \ref{1.2} does not depend on this interpretation.
\\
\\
Unfortunately our proof makes it clear that we cannot easily generalize to the case of $k(n)$-chains, i.e., where $k$ grows to $+\infty$ with $n$. As stated in the introduction, we actually expect that if $k(n)$ grows slowly enough (at a rate of $o(n^{1/2})$) then one has the same limit shape. One the other hand if $n^{1/2} = o(k(n))$, then we expect the limit to be the LSVK curve \cite{LS77,VK77}. We expect a nontrivial crossover when $k(n) \sim \alpha n^{1/2}$, because this is precisely the minimal growth rate at which the typical ensemble of sub-paths no longer has a tendency to just concentrate near the boundary of the partition, but actually distributes itself throughout the bulk of the partition according to some density (as can be shown via a random matrix argument, or alternatively using variational principles for domino tilings). This may or may not be pursued in a future work, but we believe that a similar overall approach will work.

\end{document}